\documentclass{amsart}
\usepackage{biblatex}
\usepackage{hyperref,amssymb,mathtools,multirow}

\theoremstyle{plain}
\newtheorem{thm}{Theorem}
\newtheorem{lem}{Lemma}
\theoremstyle{definition}
\newtheorem{defn}{Definition}

\providecommand{\N}[0]{\mathbb{N}}
\providecommand{\nc}[1]{\operatorname{Nc}(#1)}

\begin{document}
\title{Nonnormal quotients}
\author{Charlotte Aten}
\date{2016}
\thanks{email: \href{mailto:caten2@u.rochester.edu}{caten2@u.rochester.edu}\\I would like to thank C. Douglas Haessig for providing some helpful advice on this project.}
\begin{abstract}
We present a natural extension of the process of taking a group quotient to arbitrary subgroups. We first review basic concepts from group theory. This will allow us to see the relationship between our new, more general quotient operation and the standard group quotient. In particular, we will find that a naive attempt to perform the quotient process with a nonnormal subgroup actually leads to a well-defined operation which can be easily characterized with existing terminology.
\end{abstract}
\maketitle

\tableofcontents

\section{Groups}
We recall the relevant terminology from group theory\cite{ash,gallian}. A \emph{group} consists of an underlying set $G$ and an operation $f:G^2\to G$ such that the following conditions hold:
\begin{enumerate}
\item[Associativity] for every $x,y,z\in G$ we have that $f(f(x,y),z)=f(x,f(y,z))$,
\item[Identity] there exists $e\in G$ such that $f(x,e)=f(e,x)=x$ for every $x\in G$, and
\item[Inverses] for every $x\in G$ there exists some $x^{-1}\in G$ such that $$f(x,x^{-1})=f(x^{-1},x)=e.$$
\end{enumerate}
We often write $xy$ to indicate $f(x,y)$ when the group in question is understood. Infix notation, such as $x*y$, is also sometimes used. We will refer to both the set $G$ and the group consisting of $G$ under $f$ as simply $G$ when the context is clear.

Given a group $G$ we may have that some $H\subset G$ also forms a group under the same operation as $G$. In this case we say that $H$ is a \emph{subgroup} of $G$ and write $H\le G$. If $H\le G$ and $a\in G$ we refer to $aH\coloneqq\{ah|h\in H\}$ as the \emph{left coset} of $H$ containing $a$. This terminology is justified, for the left cosets of $H$ in $G$ partition $G$. We can similarly define the \emph{right coset} of $H$ containing $a$ as $Ha\coloneqq\{ha|h\in H\}$. The right cosets of $H$ in $G$ also partition $G$.

If a left coset $H'$ of $H\le G$ can be written as $H'=aH$ for $a\in G$, we say that $a$ is a \emph{representative} of $H'$. Any element of $H'$ can serve as a representative, and we have that $a,b\in G$ are both representatives of $H'$ if and only if $b^{-1}a\in H$. (To see this, set $aH=bH$ and note that $b^{-1}ae=b^{-1}a$ must belong to $b^{-1}aH=H$.)

In general it is not the case that $aH=Ha$ for a given $H\le G$ and $a\in G$. Moreover, the collection of left cosets of $H$ in $G$ and the collection of all right cosets of $H$ in $G$ are distinct from each other in general. If we do have that $aH=Ha$ for all $a\in G$, we say that the subgroup $H$ is \emph{normal} in $G$ and write $H\trianglelefteq G$.

If $H\trianglelefteq G$ then the (left) cosets of $H$ from a group under elementwise multiplication. The identity of this group is the coset $eH=H$. We refer to such a group as the \emph{quotient group} of $G$ by $H$, which we indicate by $G/H$.

If $H\not\trianglelefteq G$ (that is, $H$ is not normal in $G$ or $H$ is \emph{nonnormal} in $G$) then the left cosets of $H$ do not form a group. The set of such cosets is not even closed under multiplication. We now generalize the construction of a quotient group to taking a quotient by an arbitrary subgroup.

\section{Definitions}
In the following definitions let $G$ be a group and let $H$ be any subgroup of $G$. In particular, $H$ need not be normal in $G$. We would like the product of two left cosets of $H$ in $G$ to again be a left coset. Unfortunately this is not always the case, so we provide terminology for those subsets of $G$ which are products of cosets of $H$.

\begin{defn}[Block induced by $H$]
Let $C_1=aH$ and $C_2=bH$ be left cosets of $H\le G$. A \emph{left block} $B$ induced by $H$ in $G$ is a set of the form $$B\coloneqq C_1C_2=\{ah_1bh_2|h_1,h_2\in H\}$$ where $a$ and $b$ are representatives of $C_1$ and $C_2$, respectively. Right blocks are defined analogously.
\end{defn}

From now on we will only make use of left blocks and the related left-handed objects. We define representatives for blocks in analogy with those for cosets. That is, we say that $a,b$ is a \emph{representative pair} for the block $B$ if $B=aHbH$. The left blocks of $H$ in $G$ naturally induce a well-defined relation on the left cosets of $H$ in $G$ as well as a relation on the elements of $G$ itself.

\begin{defn}[Relation $\theta$]
Define a relation $\theta$ on the left cosets of $H$ in $G$ by $aH\theta bH$ for $a,b\in G$ if there exist $m_1,n_1,m_2,n_2\in G$ such that $m_1n_1=a$, $m_2n_2=b$, $m_1H=m_2H$, and $n_1H=n_2H$. That is, we say $aH\theta bH$ when there exists a block $B$ such that $a,b\in B$.
\end{defn}

\begin{lem}
The relation $\theta$ is well-defined.
\end{lem}

\begin{proof}
Suppose $aH=cH$ and $bH=dH$. We show that $aH\theta bH$ if and only if $cH\theta dH$.

Assume $aH\theta bH$. Then there exist $m_1,n_1,m_2,n_2\in G$ such that $m_1n_1=a$, $m_2n_2=b$, $m_1H=m_2H$, and $n_1H=n_2H$. Note that $c=m_1(m_1^{-1}c)$ and $d=m_2(m_2^{-1}d)$. We already have that $m_1H=m_2H$, so if we can show that $m_1^{-1}cH=m_2^{-1}dH$ we will have shown that $cH\theta dH$, as desired. This is indeed the case, as $$m_1^{-1}cH=m_1^{-1}aH=m_1^{-1}m_1n_1H=n_1H=n_2H=m_2^{-1}m_2n_2H=m_2^{-1}bH=m_2^{-1}dH.$$

By symmetry the converse holds, as well.
\end{proof}

\begin{defn}[Relation $\psi$]
Define a relation $\psi$ on the elements of $G$ by $a\psi b$ for $a,b\in G$ if there exist $x_1,y_1,x_2,y_2\in G$ such that $x_1y_1=a$, $x_2y_2=b$, $x_1H=x_2H$, and $y_1H=y_2H$. That is, we say $a\psi b$ when there exists a block $B$ such that $a,b\in B$.
\end{defn}

\begin{lem}
For any $a,b\in G$ we have that the following are equivalent:
\begin{enumerate}
\item[1.] there exists a block $B$ such that $aH,bH\subset B$
\item[2.] $aH\theta bH$
\item[3.] $a\psi b$
\end{enumerate}
\end{lem}

\begin{proof}
It follows immediately from the definitions of $\theta$ and $\psi$ that we have $a\psi b$ if and only if $aH\theta bH$. It is also immediate that if there exists a block $B$ such that $aH,bH\subset B$ then $a,b\in B$ and hence $a\psi b$ and $aH\theta bH$. Now suppose instead that there exists a block $C$ such that $a,b\in C$. We show this implies that $aH,bH\subset B$ for some block $B$.

Let $C=c_1Hc_2H$. Then $$aH\subset c_1Hc_2HH=c_1Hc_2H=C$$ and $$bH\subset c_1Hc_2HH=c_1Hc_2H=C,$$ so we can just take $B=C$.
\end{proof}

The relation $\psi$ (and hence the relation $\theta$) is symmetric and reflexive but not transitive. For example, if we take $G$ to be the symmetric group on $\{0,1,2,3\}$ and take $H$ to be the subgroup generated by the transposition $(3,4)$ then one can verify that $()\psi(1,2)$ and $(1,2)\psi(1,2,3,4)$, but it is not the case that $()\psi(1,2,3,4)$ as there is no block which contains both $()$ and $(1,2,3,4)$.

\section{Main Result}
We again let $G$ be a group with a subgroup $H$. We denote the normal closure of $H$ by $\nc{H}$ and denote the identity in $G$ by $e$. We now take the transitive closure of the relation $\psi$ on the elements of $G$. Let $S_0=H$ and for any $n\in\N$ let $$S_n=\{g\in G|g\psi s\text{ for some }s\in S_{n-1}\}.$$ Also let $S=\bigcup_{n\in\N}S_n$.

\begin{defn}
Define a relation $\sim$ on $G$ as follows. Given $g_1,g_2\in G$ we say that $g_1\sim g_2$ when there exists some $a\in G$ such that $g_1,g_2\in aS$.
\end{defn}

\begin{thm}
The relation $\sim$ is precisely the relation induced by the left cosets of $\nc{H}$ in $G$. Moreover, $S=\nc{H}$.
\end{thm}

\begin{proof}
We will show that $S=\nc{H}$, so that our definition of $\sim$ is well-defined. Recall that $\nc{H}=\langle ghg^{-1}|g\in G,h\in H\rangle$. We begin by showing that $$S_1\coloneqq\{g\in G|g\psi h\text{ for some }h\in H\}\subset\langle ghg^{-1}|g\in G,h\in H\rangle=\nc{H}.$$

Suppose that $g\psi h$. Then $g,h\in aHbH$ for some $a,b\in G$, so $g=ah_1bh_2$ and $h=ah_3bh_4$ for some $h_1,h_2,h_3,h_4\in H$. It follows that $$b=h_3^{-1}a^{-1}hh_4^{-1}=h_3^{-1}a^{-1}h_5$$ for some $h_5\in H$. Then $$g=ah_1(h_3^{-1}a^{-1}h_5)h_2=a(h_1h_3^{-1})a^{-1}eh_2e^{-1}=(ah_6a^{-1})(eh_2e^{-1})$$ where $h_6\in H$. Since $a,e\in G$, we have written $g$ as a product of two of the generators of $\nc{H}$. Thus, $$S_1=\{g\in G|g\psi h\text{ for some }h\in H\}\subset\langle ghg^{-1}|g\in G,h\in H\rangle=\nc{H}.$$

Now instead suppose that $g_1$ belongs to the set of canonical generators for $\nc{H}$ where $g_1=g_3h_1g_3^{-1}$ for $g_3\in G$ and $h_1\in H$. Suppose also that $g_2\in G$. Note that $$g_2=g_3eg_3^{-1}g_2e,$$ so $g_2\in g_3Hg_3^{-1}g_2H$. Note also that $$g_1g_2=g_3h_1g_3^{-1}g_2e,$$ so $g_1g_2\in g_3Hg_3^{-1}g_2H,$ as well. It follows that $g_2\psi g_1g_2$.

Consider the element $g\in\nc{H}$ where $g=\prod_{i=1}^{k}g_i$ with each $g_i$ of the form $\gamma_ih_i\gamma^{-1}_i$ for some $\gamma_i\in G$ and some $h_i\in H$. Since $g_k\in \gamma_kH\gamma^{-1}_kH$ and $e=\gamma_ke\gamma^{-1}_ke\in \gamma_iH\gamma^{-1}_kH$, we have that $g_k\psi e$ and thus $g_k\in S_1$.

By our previous reasoning it follows that $g_k\psi g_{k-1}g_k$, so $g_{k-1}g_k\in S_2$. By induction we see that $g=\prod_{i=1}^{k}g_i\in S_k\subset S$. Since for any $g\in\nc{H}$ we have that $g\in S$, it follows that $\nc{H}\subset S$.

We now show that $S\subset\nc{H}$. Note that the left cosets of $\nc{H}$ partition the left cosets of $H$ as well as the elements of $G$. It follows that the blocks induced by $\nc{H}$ partition the blocks induced by $H$. We now assume towards a contradiction that $S\not\subset\nc{H}$.

Since we assume $S\supsetneq\nc{H}$ there exist some $a,b\in S$ such that $a\psi b$ with $a\in\nc{H}$ and $b\notin\nc{H}$. By definition of $\sim$ there must then exist blocks $B_1,B_2$ induced by $H$ such that $B_1\psi B_2$ where $a\in B_1$ and $b\in B_2$. This implies that there exists some $c\in G$ such that $c\in B_1\cap B_2$.

Let $B_2=\alpha H\beta H$. Then there exists some $h\in H$ such that $c\in\alpha h\beta H$ with $\alpha h\beta H\not\subset\nc{H}$. Since the cosets of $\nc{H}$ partition $G$ and $\alpha h\beta H\not\subset\nc{H}$, this is a contradiction, as the element $c$ belongs to both $\nc{H}$ in addition to another, distinct coset of $\nc{H}$. It must then be that $S\subset\nc{H}$.

As we already had containment in the other direction, this establishes that $S=\nc{H}$ and that the relation $\sim$ is well-defined.

Since we now know that $S=\nc{H}$, we have that $g_1\sim g_2$ for $g_1,g_2\in G$ if and only if there exists some $a\in G$ such that $g_1,g_2\in a\nc{H}$. This is the definition of the relation induced on the elements of $G$ by the left cosets of $\nc{H}$.
\end{proof}

The generalized quotient of $G$ by $H$ may then be defined as the group of equivalence classes of elements of $G$ under the relation $\sim$ induced by $H$ as above. We have already seen that our this operation is natural, for we began examining blocks induced by $H\le G$ and made the innocent identification of two blocks with a nonempty intersection. The described completion of this relation is precisely that induced by (left) cosets of $\nc{H}$ on the elements of $G$. It is now apparent that the appropriate extension of quotients of a group $G$ to nonnormal subgroups $H$ is merely taking the quotient $G/\nc{H}$. Of course, this agrees with the usual quotient construction when $H\trianglelefteq G$.

\section{The Block Relation $\rho$}
Again take $G$ to be a group and take $H$ to be a subgroup of $G$, in particular a nonnormal subgroup. One might wonder what happens if, upon discovering that the product of two cosets of $H$ was not always a coset, we make the following definition.

\begin{defn}(Relation $\rho$)
Let $B_1$ and $B_2$ be left blocks induced by $H\le G$. We say that $B_1\rho B_2$ if $B_1\cap B_2\neq\varnothing$.
\end{defn}

Essentially we would like to ``glue together'' blocks induced by $H$ until we have a collection of disjoint subsets of $G$ which form a group under elementwise multiplication. This was, in fact, the original motivation for the present paper, but the previous argument in terms of the relation $\theta$ on individual group elements turned out to be easier to produce. Based on that argument we now know that such a ``gluing'' process must ultimately yield an equivalence relation on blocks induced by $H$ which corresponds to the relation induced by $\nc{H}$ on the elements of $H$.

As we saw with the relations $\theta$ and $\psi$, it is trivial that $\rho$ is reflexive and symmetric, but it is not in general transitive\footnote{Thanks to Michael Kinyon for pointing out that we already had examples where transitivity fails.}. For example, if we take $G$ to be the symmetric group on $\{0,1,2\}$ and take $H$ to be the subgroup generated by the transposition $(2,3)$ then one can verify that $HH\rho(1,2)H(1,2)H$ and $(1,2)H(1,2)H\rho H(1,2)H$, but it is not the case that $HH\rho H(1,2)H$ as the cosets of $H$ in $G$ are disjoint. If $\rho$ is transitive then we have the following result.

\begin{lem}
Suppose $\rho$ is transitive. Then the elements of the blocks $B$ such that $B\rho H$ are precisely the elements of $\nc{H}$.
\end{lem}

\begin{proof}
Assume $B\rho H$. Then given any $b\in B$ we have that $b\psi h$ for some $h\in H$ since $B\rho H$ implies that $B\cap H\neq\varnothing$ and thus $b,h\in B$. By our main result we know that $b\in\nc{H}$, so certainly $B\subset\nc{H}$ for every $B\rho H$.

Now suppose we have some block $C$ with $c\in C$ such that $c\in\nc{H}$. Then again by our main result we know that $$c\psi b_k\psi b_{k-1}\psi\dots\psi b_1\psi h$$ for some $b_i$ such that $b_i\in B_i$ for blocks $B_i$. Then by definition of $\psi$ we have that $$C\rho B_k\rho B_{k-1}\rho\cdots\rho B_1\rho H$$ so by transitivity of $\rho$ we have $C\rho H$. It follows that the blocks $B$ such that $B\rho H$ are precisely the elements of $\nc{H}$.
\end{proof}

In other words, if $\rho$ is transitive then $\nc{H}=\bigcup\{B|B\cap H\neq\varnothing\}$ for any $H\le G$.

\section{Example Computations}
The study of this generalized quotient process was heavily aided by the examination of examples drawn up by scripts using SageMath (\url{http://www.sagemath.org/}). These scripts can be found at \url{https://github.com/caten2/NonnormalQuotient}. Some sample output follows.

\begin{figure}[h]
\caption{Quotient of $S_3$ by $H=\langle(1,2)\rangle$.}
\begin{tabular}{| *{3}{r|} *{2}{c} | *{2}{c} | *{2}{c} | } \cline{4-9}
\multicolumn{3}{c|}{} & \multicolumn{6}{c|}{$\nc{H}$} \\ \cline{4-9}
\multicolumn{3}{c|}{} & \multicolumn{2}{c|}{$H$} & \multicolumn{2}{c|}{$(2,3)H$} & \multicolumn{2}{c|}{$(1,3,2)H$} \\ \cline{4-9}
\multicolumn{3}{c|}{} & $()$ & $(1,2)$ & $(2,3)$ & $(1,2,3)$ & $(1,3,2)$ & $(1,3)$ \\ \hline
\multirow{6}{*}{$\nc{H}$} & \multirow{2}{*}{$H$} & $()$ & $()$ & $(1,2)$ & $(2,3)$ & $(1,2,3)$ & $(1,3,2)$ & $(1,3)$ \\
& & $(1,2)$ & $(1,2)$ & $()$ & $(1,3,2)$ & $(1,3)$ & $(2,3)$ & $(1,2,3)$ \\ \cline{2-9}
& \multirow{2}{*}{$(2,3)H$} & $(2,3)$ & $(2,3)$ & $(1,2,3)$ & $()$ & $(1,2)$ & $(1,3)$ & $(1,3,2)$ \\
& & $(1,2,3)$ & $(1,2,3)$ & $(2,3)$ & $(1,3)$ & $(1,3,2)$ & $()$ & $(1,2)$ \\ \cline{2-9}
& \multirow{2}{*}{$(1,3,2)H$} & $(1,3,2)$ & $(1,3,2)$ & $(1,3)$ & $(1,2)$ & $()$ & $(1,2,3)$ & $(2,3)$ \\
& & $(1,3)$ & $(1,3)$ & $(1,3,2)$ & $(1,2,3)$ & $(2,3)$ & $(1,2)$ & $()$ \\ \hline
\end{tabular}
\end{figure}

\begin{figure}[h]
\caption{Quotient of $S_3$ by $H=\langle(1,3)\rangle$.}
\begin{tabular}{| *{3}{r|} *{2}{c} | *{2}{c} | *{2}{c} | } \cline{4-9}
\multicolumn{3}{c|}{} & \multicolumn{6}{c|}{$\nc{H}$} \\ \cline{4-9}
\multicolumn{3}{c|}{} & \multicolumn{2}{c|}{$H$} & \multicolumn{2}{c|}{$(2,3)H$} & \multicolumn{2}{c|}{$(1,2)H$} \\ \cline{4-9}
\multicolumn{3}{c|}{} & $()$ & $(1,3)$ & $(2,3)$ & $(1,3,2)$ & $(1,2)$ & $(1,2,3)$ \\ \hline
\multirow{6}{*}{$\nc{H}$} & \multirow{2}{*}{$H$} & $()$ & $()$ & $(1,3)$ & $(2,3)$ & $(1,3,2)$ & $(1,2)$ & $(1,2,3)$ \\
& & $(1,3)$ & $(1,3)$ & $()$ & $(1,2,3)$ & $(1,2)$ & $(1,3,2)$ & $(2,3)$ \\ \cline{2-9}
& \multirow{2}{*}{$(2,3)H$} & $(2,3)$ & $(2,3)$ & $(1,3,2)$ & $()$ & $(1,3)$ & $(1,2,3)$ & $(1,2)$ \\
& & $(1,3,2)$ & $(1,3,2)$ & $(2,3)$ & $(1,2)$ & $(1,2,3)$ & $(1,3)$ & $()$ \\ \cline{2-9}
& \multirow{2}{*}{$(1,2)H$} & $(1,2)$ & $(1,2)$ & $(1,2,3)$ & $(1,3,2)$ & $(2,3)$ & $()$ & $(1,3)$ \\
& & $(1,2,3)$ & $(1,2,3)$ & $(1,2)$ & $(1,3)$ & $()$ & $(2,3)$ & $(1,3,2)$ \\ \hline
\end{tabular}
\end{figure}

\begin{figure}[h]
\caption{Quotient of $S_3$ by $H=\langle(2,3)\rangle$.}
\begin{tabular}{| *{3}{r|} *{2}{c} | *{2}{c} | *{2}{c} | } \cline{4-9}
\multicolumn{3}{c|}{} & \multicolumn{6}{c|}{$\nc{H}$} \\ \cline{4-9}
\multicolumn{3}{c|}{} & \multicolumn{2}{c|}{$H$} & \multicolumn{2}{c|}{$(1,2)H$} & \multicolumn{2}{c|}{$(1,2,3)H$} \\ \cline{4-9}
\multicolumn{3}{c|}{} & $()$ & $(2,3)$ & $(1,2)$ & $(1,3,2)$ & $(1,2,3)$ & $(1,3)$ \\ \hline
\multirow{6}{*}{$\nc{H}$} & \multirow{2}{*}{$H$} & $()$ & $()$ & $(2,3)$ & $(1,2)$ & $(1,3,2)$ & $(1,2,3)$ & $(1,3)$ \\
& & $(2,3)$ & $(2,3)$ & $()$ & $(1,2,3)$ & $(1,3)$ & $(1,2)$ & $(1,3,2)$ \\ \cline{2-9}
& \multirow{2}{*}{$(1,2)H$} & $(1,2)$ & $(1,2)$ & $(1,3,2)$ & $()$ & $(2,3)$ & $(1,3)$ & $(1,2,3)$ \\
& & $(1,3,2)$ & $(1,3,2)$ & $(1,2)$ & $(1,3)$ & $(1,2,3)$ & $()$ & $(2,3)$ \\ \cline{2-9}
& \multirow{2}{*}{$(1,2,3)H$} & $(1,2,3)$ & $(1,2,3)$ & $(1,3)$ & $(2,3)$ & $()$ & $(1,3,2)$ & $(1,2)$ \\
& & $(1,3)$ & $(1,3)$ & $(1,2,3)$ & $(1,3,2)$ & $(1,2)$ & $(2,3)$ & $()$ \\ \hline
\end{tabular}
\end{figure}

\printbibliography

\end{document}